\documentclass[11pt,a4paper]{article}

\usepackage{authblk}

\makeatletter
\newcommand{\subjclass}[2][2010]{%
  \let\@oldtitle\@title%
  \gdef\@title{\@oldtitle\footnotetext{#1 \emph{Mathematics subject classification.} #2}}%
}
\newcommand{\keywords}[1]{%
  \let\@@oldtitle\@title%
  \gdef\@title{\@@oldtitle\footnotetext{\emph{Key words and phrases.} #1.}}%
}
\makeatother

\AtEndDocument{\bigskip{\footnotesize%
  \textit{E-mail address}, Xiaoqin Guo: \texttt{guoxq@ucmail.uc.edu} \par
  \addvspace{\medskipamount}
  \textit{E-mail address}, Timo Sprekeler: \texttt{timo.sprekeler@tamu.edu} \par
  \addvspace{\medskipamount}
  \textit{E-mail address}, Hung Vinh Tran: \texttt{hung@math.wisc.edu}
}}

\usepackage{etoolbox}
\usepackage{comment,amsmath,amssymb,amsthm,bbm,mathtools,mathrsfs}
\usepackage{tikz,pgfplots,color,setspace,lmodern}
\usepackage[]{stix}
\usetikzlibrary{shapes, backgrounds,patterns,decorations.pathreplacing}
\usepackage{multicol,float}
\usepackage{makeidx}
\usepackage{hyperref, enumerate}
\pgfplotsset{compat=newest}
\makeindex

\def\XXint#1#2#3{{\setbox0=\hbox{$#1{#2#3}{\int}$ }
\vcenter{\hbox{$#2#3$ }}\kern-.6\wd0}}

\newcommand\error{\varepsilon}

\newcommand{\dd}{\mathrm{d}}

\newcommand{\B}{\mathbb{B}}

\newcommand{\R} {\mathbb{R}}
\newcommand{\Q} {\mathbb{Q}}
\newcommand{\Z} {\mathbb{Z}}
\newcommand{\nn}{\nonumber}

\newcommand{\tx}{\mathscr{H}}
\newcommand{\inB}[1]{B_{#1}^\circ}
\newcommand{\binB}[1]{\bar{B}_{#1}^\circ}

\newcommand{\evp}[1]{\bar\omega^{#1}}
\newcommand{\fct}{\mathrm{H}} 
\newcommand{\apx}{{\scriptscriptstyle\rm AP}}
\newcommand{\krt}{{\pmb{\phi}}}

\DeclarePairedDelimiter{\abs}{\lvert}{\rvert}

\DeclarePairedDelimiter{\norm}{\lVert}{\rVert}
\providecommand{\Abs}[1]{\Bigr\lvert#1\Bigl\rvert}

\providecommand{\mb}[1]{\mathbb{#1}}

\providecommand{\ms}[1]{\mathscr{#1}}

\DeclareMathOperator*{\osc}{osc}

\newcommand{\tr}{{\rm tr}}
\newcommand{\diag}{{\rm diag}}

\newtheorem{thmx}{Theorem}

\newtheorem{theorem}{Theorem}
\newtheorem{lemma}[theorem]{Lemma}

\newtheorem{proposition}[theorem]{Proposition}

\theoremstyle{definition}  
\newtheorem{definition}[theorem]{Definition}
\newtheorem{remark}[theorem]{Remark}

\begin{document}

\title{Homogenization of non-divergence form operators in i.i.d.  random environments}

\author[1]{Xiaoqin Guo \thanks{The work of XG is supported by Simons Foundation through Collaboration Grant for Mathematicians \#852943.}}
\author[2]{Timo Sprekeler}
\author[3]{Hung V. Tran  \thanks{HT is supported in part by NSF grant DMS-2348305.}}

\affil[1]
{
Department of Mathematical Sciences, 
University of Cincinnati, 2815 Commons Way, Cincinnati, OH 45221, USA}
\affil[2]
{Department of Mathematics, Texas A{\&}M University, College Station, TX 77843, USA}

\affil[3]
{
Department of Mathematics, 
University of Wisconsin-Madison, 480 Lincoln  Drive, Madison, WI 53706, USA}



\maketitle

\begin{abstract}
We study random walks in a balanced, i.i.d. random environment in $\Z^d$ for $d\geq 3$.
We establish improved convergence rates for the homogenization of the Dirichlet problem associated with the corresponding non-divergence form difference operators, surpassing the $O(R^{-1})$ rate, which is expected to be optimal for environments with a finite range of dependence.
In particular, the improved rates are $O(R^{-3/2})$ when $d=3$, and $O(R^{-2}\log R)$ when $d\geq 4$.
\end{abstract}


\section{Introduction}\label{sec:intro}
In this paper, we consider the homogenization convergence rates of purely second-order non-divergence form difference operators in an independent and identically distributed (i.i.d.) random environment on $\Z^d$ for dimensions $d\ge 3$.

Large-scale behavior of random walks in i.i.d.\ environments has been intensively studied; see \cite{BoSz-02, OZ-04, MB-11, DreRam-14, Kumagai-14} and the references therein. Since the dynamics of random walks in random environments (RWRE) can be described by difference equations, stochastic homogenization of difference equations in random environments is crucial for understanding the large-scale behavior of RWRE.
For environments in which spectral-gap or Efron--Stein type inequalities are available, the strong concentration properties of the medium can be exploited to obtain quantitative homogenization results; see, for example, \cite{GNO-15, AL-17, GT-23}. In this paper, we show that, for non-divergence form difference operators in the i.i.d.\ setting, the reflection symmetry of the law of the environment can in fact be used to extract additional cancellations. Somewhat unexpectedly, this yields homogenization rates for the Dirichlet problem that are strictly better than the standard optimal rate $O(R^{-1})$ obtained in \cite{GT-23}. To achieve this, we require a more delicate analysis of the fluctuations of the homogenization error, relying in particular on a deeper understanding of higher-order correctors.

\subsection{Settings}
Let $\mb S_{d\times d}$ denote the set of real $d\times d$ positive definite diagonal matrices. A map 
\[
\omega:\Z^d\to\mb S_{d\times d}
\] is called an {\it environment}, and the set of all such environments is denoted by $\Omega$. 
Let $\mb P$ be a probability measure on $\Omega$ such that 
\[
\left\{\omega(x)=\mathrm{diag}[\omega_1(x),\ldots, \omega_d(x)], \;x\in\Z^d\right\}
\] 
are i.i.d. under $\mb P$. We denote the expectation with respect to $\mb P$ by $\mb E$ (or $E_{\mb P}$). 

\begin{definition}\label{def:differences}
Let $\{e_1,\ldots,e_d\}$ be the canonical basis for $\R^d$, and let 
\[
U:=\left\{e\in\Z^d:|e|=1\right\}=\left\{\pm e_1,\ldots,\pm e_d\right\}
\]
be the set of unit vectors in $\Z^d$. We define the difference operators $\nabla=(\nabla_e)_{e\in U}$  and 
$\nabla^2= \mathrm{diag}[\nabla_1^2,\ldots,\nabla_d^2]$ by
\begin{equation}\label{eq:def-nabla}
\nabla_e u(x) := u(x+e)-u(x), \qquad
\nabla_i^2u(x) := u(x+e_i)+u(x-e_i)-2u(x)  
\end{equation}
for $e\in U$ and $i\in\{1,\ldots,d\}$. 
Clearly, $\nabla$ and $\nabla^2$ are linear operators.  
\end{definition}

For $r>0$ and $y\in\R^d$, we write
\[
\B_r(y) := \left\{x\in\R^d: |x-y|<r\right\}, \qquad
B_r(y):=\B_r(y)\cap\Z^d
\]
to denote the continuous and discrete ball with center $y$ and radius $r$, respectively.
When $y=0$, we write $\B_r := \B_r(0)$ and $B_r := B_r(0)$ for simplicity. 
For any $B\subset\Z^d$, its {\it discrete boundary} $\partial B$ and {\it discrete closure} $\bar B$ are defined as
\[
\partial B:=\left\{z\in\Z^d\setminus B: |z-x|=1 \text{ for some }x\in B\right\},\qquad \bar B := B\cup\partial B.
\]
By abuse of notation, we also use $\partial A$ and $\bar A$ to denote the usual continuous boundary and closure of $A\subset\R^d$, respectively, if there is no confusion.  
The interior of $B_R$ is denoted by
\[
\inB{R} := \{x\in B_R\,:\,  |x-y| \ge 1 \text{ for all }y\in\partial\B_R\}.
\]
Note that for any $R\ge 1$, there holds
\[
\bar{\B}_{R-1}\cap\Z^d\subset\inB 
R\subset B_R\subset\binB{R}\subset\bar\B_R\cap\Z^d\subset\bar B_R.
\]

 For $x\in\Z^d$, a {\it spatial shift} is defined by 
 \[
 \theta_x:\Omega\to\Omega,\qquad \theta_x\omega := \omega(x+\cdot).
 \] 
In a random environment $\omega\in\Omega$ in dimensions $d\ge 2$,  we consider the discrete {\it non-divergence form}   elliptic Dirichlet problem
\begin{equation}\label{eq:elliptic-dirich}
\left\{
\begin{array}{lr}
\tfrac 12\tr\big(\omega(x)\nabla^2u(x)\big)=\frac{1}{R^2}f\left(\tfrac{x}{R}\right)\zeta(\theta_x\omega) & x\in \inB R,\\[5 pt]
u(x)=g\left(\tfrac{x}{R}\right) & x\in \partial \inB R,
\end{array}
\right.
\end{equation}
where $f,g\in\R^{\bar\B_1}$ are sufficiently regular, and $\zeta\in \R^\Omega$ satisfies suitable integrability conditions.
As $R\to\infty$, it is expected that $\vert u - \bar u(\frac{\cdot}{R})\rvert$ converges to $0$, where $\bar u$ solves 
\begin{equation}\label{eq:effective-ellip}
\left\{
\begin{array}{lr}
\tfrac 12\tr(\bar a\, D^2\bar{u}(x))
=f(x)\,\bar\psi &x\in\B_1,\\ 
\bar u(x)=g(x) &x\in\partial \B_1.
\end{array}
\right.
\end{equation}
Here, $D^2 \bar{u}$ denotes the Hessian matrix of $\bar{u}$, and $\bar a=\bar a(\mb P)\in\mb S_{d\times d}$ and $\bar\psi=\bar\psi(\mb P,\zeta)\in\R$ are {\it deterministic} and do not depend on the realization of the random environment; see Theorem~\ref{thm:QCLT} and Theorem~\ref{thm:opt-quant-homo} for formulas of $\bar{a}$ and $\bar{\psi}$.

\medskip

We remark that in the continuous PDE setting, one only needs a function $g\in C^{2,\alpha}(\partial\B_1)$ on the boundary $\partial\B_1$ to consider Dirichlet problems of types \eqref{eq:elliptic-dirich} and \eqref{eq:effective-ellip}. However, in the discrete setting, 
the discrete and continuous boundaries do not agree, and therefore the function $g$ in $\eqref{eq:elliptic-dirich}$ should be extended to $\frac{1}{R}\partial\inB R\not\subset\partial\B_1$.  This discrepancy typically induces an error of order $\tfrac{1}{R}$ which automatically becomes the ceiling of the optimal homogenization error in the discrete setting,  cf.  \cite{GT-23}. 
Indeed, this $O(R^{-1})$ rate is the optimal rate obtained in \cite{GT-23} for $d\geq 3$.
Whereas, when such a discretization issue is not present, e.g., in the continuous periodic setting \cite{GTY-19, ST-21, Spr-24},  one needs to go deeper into higher order correctors to justify the optimality of the rate $O(R^{-1})$ for the homogenization of the Dirichlet problem. We refer to \cite{CSS-20, FGP-24, SSZ-25} for recent developments in numerical homogenization of PDEs in non-divergence form.

\medskip

A goal of this article is to isolate and examine the homogenization rate that does {\bf not} originate from boundary discretization.
To this end, we assume that the boundary data $g$ is properly extended to be a function on $\frac{1}{R}\partial\inB R$ in the Dirichlet problem \eqref{eq:elliptic-dirich}.

\medskip

The non-divergence form difference equation \eqref{eq:elliptic-dirich} is used to describe random walks in a random environment (RWRE) in $\Z^d$. To be specific, for $x,y\in \Z^d$, we set
\begin{equation}\label{def:omegaxei}
\omega(x,x\pm e_i):=\frac{\omega_i(x)}{2\,\tr\omega(x)} \quad  \text{ for } i=1,\ldots, d,
\end{equation}
 and 
 \[
 \omega(x,y) := 0 \quad \text{ if $|x-y|\neq 1$},
 \]
 that is, we normalize $\omega$ to get a transition probability.  
 We note that the configuration of $\{\omega(x,y):x,y\in\Z^d\}$ is also called a {\it balanced environment} in the literature \cite{L-82,GZ-12,BD-14,DGR-15, BCDG-18,DG-19}.
\begin{definition}\label{def:rwre-discrete}
For each fixed $\omega\in\Omega$, the random walk  $(X_n)_{n\ge 0}$ in the environment $\omega$ with $X_0=x$ is a Markov chain  in $\Z^d$ with transition probability $P_\omega^x$ specified by
\begin{equation}\label{eq:def-RW}
P_\omega^x\left(X_{n+1}=z|X_n=y\right) := \omega(y,z).
\end{equation}
The {\it continuous-time} RWRE $(Y_t)_{t\ge 0}$ in $\omega$ is defined to be the Markov process on $\Z^d$ with jump rates $\omega(x,y)$ from $x$ to $y$. With abuse of notations, we still denote by $P_\omega^x$  the quenched law of $(Y_t)$ if there is no ambiguity from the context.
\end{definition} 
The expectation with respect to $P_\omega^x$ is denoted by $E_\omega^x$. When the starting point of the random walk is $0$, we sometimes omit the superscript and write $P_\omega^0$ and $E_\omega^0$ as $P_\omega$ and $E_\omega$, respectively. Note that for random walks $(X_n)$ in an environment $\omega$, we have that
\begin{equation}\label{def:omegabar}
\bar\omega^i := \theta_{X_i}\omega\in\Omega, \qquad i\ge 0,
\end{equation}
is also a Markov chain, called the {\it environment viewed from the particle} process. By abuse of notation, we enlarge our probability space so that $P_\omega$ still denotes the joint law of the random walks and $(\bar\omega^i)_{i\ge 0}$. 

For $\omega\in\Omega$ and $i\in \{1,\ldots,d\}$, we set $a_i(x)=a_i^\omega(x) := \frac{\omega_i(x)}{\tr\omega(x)} = 2\omega(x,x+e_i)$ and
\begin{equation}
\label{eq:def-a}
a(x)=a^\omega(x):=\diag[a_1(x),\ldots, a_d(x)]=\frac{\omega(x)}{\tr\omega(x)}, \quad x\in\Z^d.
\end{equation}
Finally, for $\omega\in \Omega$, we introduce the difference operator $L_{\omega}$ given by
\begin{align*}
  L_\omega u(x) := \sum_{y\in \Z^d}\omega(x,y)[u(y)-u(x)]=\tfrac{1}{2}\tr(a(x)\nabla^2 u(x)).  
\end{align*}
That is, $L_\omega$ is the generator of the process $(Y_t)_{t\ge 0}$.

\subsection{Main assumptions and some notations}
Throughout the paper, the following assumptions are always in force:
\begin{enumerate}[(\text{A}1)]
\item $\left\{\omega(x),\, x\in \Z^d \right\}$ are i.i.d. under the probability measure $\mb P$.
\item\label{item:ellip-cond} $\frac{\omega}{\tr\omega}\ge 2\kappa{\rm I}$ for $\mb P$-almost every $\omega$ and some constant $\kappa\in(0,\tfrac{1}{2d}]$.
\item $\psi(\omega):=\tfrac{\zeta(\omega)}{\tr\omega(0)}$ is a $L^\infty(\mb P)$-bounded measurable function of $\omega(0)$.
\end{enumerate} 
Throughout this paper, we use $c, C$ to denote positive constants which may change from line to line but only depend on the dimension $d$ and the ellipticity constant $\kappa$ unless stated otherwise. We write $A
\lesssim B$ if $A\le CB$, and $A\asymp B$ if $A\lesssim B$ and $A\gtrsim B$. 
We also use the notations $A\lesssim_j B$ and $A\asymp_j B$ to indicate that the multiplicative constant depends on the variable $j$ other than $(d,\kappa)$.
We will use the convention of summation over repeated indices.  
For example, $a_kb_k$ represents $\sum_{k=1}^d a_kb_k$.
This convention is not used in definitions \eqref{eq:def-lambda} and \eqref{eq:def-eta}. 
For simplicity, we use $\tx=\tx(\omega)$ to denote a generic {\it stretched-exponentially integrable} random variable (i.e., $\mb E[\exp(\tx^c)]<C$) whose value may differ from line to line. 
We write $\tx_x(\omega):=\tx(\theta_x\omega)$. 
We note that $\tx$ may depend on other parameters as well, although the constants $c, C$ in the stretched-exponential moment bound only depend on $(d,\kappa)$.

\subsection{Earlier results in the literature}

We first recall the following quenched central limit theorem (QCLT) proved in \cite{L-82}, which is a discrete version of \cite{PV-82}.
\begin{thmx}\label{thm:QCLT}
Assume \rm{(A2)} and that law $\mb P$ of the environment is ergodic under spatial shifts $\{\theta_x:x\in\Z^d\}$. Then, the following assertions hold. 
\begin{enumerate}[(i)]
\item\label{item:ergodic} There exists a probability measure $\mb Q\approx\mb P$ such that $(\evp{i})_{i\ge 0}$ is an ergodic (with respect to time shifts) sequence under law $\mb Q\times P_\omega$.
\item For $\mb P$-almost every $\omega$, the rescaled path $X_{n^2t}/n$ converges weakly (under law $P_\omega$) to a Brownian motion with covariance matrix 
\begin{equation}
\label{eq:def-abar}
\bar a=\diag[\bar a_1,\ldots,\bar a_d]:=E_\Q[a]=E_\Q[\tfrac{\omega(0)}{\tr\omega(0)}]>0.
\end{equation}
\end{enumerate}
\end{thmx}
We denote the Radon--Nikodym derivative of $\Q$ with respect to $\mb P$ by
\begin{equation}\label{eq:def-rho}
\rho(\omega)=\dd\mb Q/\dd\mb P.
\end{equation}
If the environment is i.i.d., it was proved in \cite[Theorem 1.5]{GT-22} that 
\begin{equation}
    \label{eq:rho-bounds}
    E[\exp(\rho^{-c})]+E[\exp(\rho^{c})]<\infty
\end{equation}
for some constant $c=c(\kappa,d)>0$.
For any $x\in\Z^d$, we write 
\[
\rho_\omega(x):=\rho(\theta_x\omega).
\]
For any $\alpha\in(0,1]$ and any function $f$ on a set $A$, we define
\[
\osc_A f:=\sup_{x,y\in A}|f(x)-f(y)|,
\qquad
[f]_{\alpha;A} :=\sup_{x,y\in A,\; x\neq y}\frac{|f(x)-f(y)|}{|x-y|^\alpha},
\]
and, if $A$ is a finite set, for $p\in(0,\infty)$, we define 
\[
\norm{f}_{p;A}:=\left(\frac{1}{\#A}\sum_{x\in A}|f(x)|^p\right)^{1/p}, \quad
\norm{f}_{\infty;A} :=\max_{x\in A}|f(x)|.
\]

For any integer $j\ge 0$, let $\fct_j$ denote the set of $j$-th order polynomials, with $\fct_0=\R$. In fact, in our paper, we will only use the cases $j=0,1$.

The following large-scale regularity result can be found in \cite{GT-23}, which is a discrete version of the $C^{1,1}$ regularity estimate in \cite[Theorem 3.1]{AL-17}.
This was first done in the periodic setting in \cite{AL-87,AL-89}.

\begin{thmx}
\label{thm:c11}
Assume {\rm(A1), (A2)},  and that $\psi$ is a {\bf local} function. Let $R\ge 1$.  
There exists $\alpha=\alpha(d,\kappa)\in(0,\tfrac13)$ such that,
for any $u$ with $L_\omega u(x)=\psi(\theta_x\omega)+f(x)$ on $B_R$, $j\in\{1,2\}$, $\tx\le r<R$, we have that 
\begin{equation}
\label{eq:oscillation-estimates}
\frac{1}{r^{j}}\inf_{p\in\fct_{j-1}}\osc_{B_r}\,(u-p)
\lesssim
\frac{1}{R^{j}}
\inf_{p\in\fct_{j-1}}\osc_{B_R}\,(u-p)+A_{j},
\end{equation}
where the terms $A_1=A_1(R)$ and $A_2=A_2(R,r)$ obey the bounds
\begin{align*}
&A_1\le R^{1-\alpha}\norm{\psi}_\infty+R\norm{f}_\infty
\quad\text{ and }\quad
A_1\le R^{1-\alpha}\norm{\psi+f(0)}_\infty+R^{1+\sigma}[f]_{\sigma;B_R},
\\
&A_2\le r^{-\alpha}\norm{\psi}_\infty+\log(\tfrac{R}{r})\norm{f}_\infty
\quad\text{ and }\quad
A_2\le r^{-\alpha}\norm{\psi+f(0)}_\infty+R^\sigma[f]_{\sigma;B_R}
\end{align*}
for any $\sigma\in(0,1]$. In particular, for any $R>1$ and $j\in\{1,2\}$, there holds
\begin{equation}\label{eq:c2}
|\nabla^ju(0)|\lesssim \left(\frac{\tx}{R}\right)^j 
\left(
\norm{u-u(0)}_{1;B_R}+R^2\norm{\psi+f(0)}_\infty+R^{2+\sigma}[f]_{\sigma;B_{R}}
\right).
\end{equation}
\end{thmx}

Introducing the quantities
\begin{equation}\label{eq:def-mu}
\mu(R):=
\left\{
\begin{array}{lr}
R & d=2,\\
R^{1/2} &d=3,\\
(1\vee \log R)^{1/2} &d=4,\\
1 &d\ge 5,
\end{array}
\right.
\quad
\delta(R):=\left\{
\begin{array}{lr}
(\log (R\vee 2))^{3/2} &d=2,\\
1 &d\ge 3,
\end{array}
\right.
\end{equation}
the following theorems on properties of the (global) correctors and quantitative homogenization of the Dirichlet problem were established in \cite{GT-23}.

\begin{thmx}
\label{thm:global_krt}
Let $\psi$ be an $L^\infty(\mb P)$-bounded function of $\omega(0)$ with $\norm{\psi}_\infty=1$.
For each $d\ge 2$ and $\mb P$-a.e. $\omega$, there exists a function $\krt=\krt_\omega:\Z^d\to\R$ that solves 
\[
L_\omega\krt(x)=\psi(\theta_x\omega)-\bar\psi \quad\text{ for }x\in\Z^d
\]
with the following properties:
\begin{enumerate}[(i)]
\item\label{item:gkrt-2} When $d\ge 5$,\quad $\mb E[\exp(c|\krt(x)/\mu(|x|)|^p)]<C_p$ for any $p\in(0,\tfrac{2d}{3d+2}),x\in\Z^d$;

When $d=3,4$,\; $\mb E[\exp(c|\krt(x)/\mu(|x|)|^p)]<C_p$ for any $p\in(0,\tfrac{2d}{3d+4}),x\in\Z^d$;

When $d=2$,\quad
$\mb E\left[
\exp\big(
c\abs{\tfrac{\krt(x)}{|x|\delta(|x|)}}^p
\big)
\right]\le C_p$ for any $p\in(0,\tfrac13), x\in\Z^d$;

\item\label{item:gkrt-3} $\mb E[\exp(c|\nabla\krt(x)/\delta(|x|)|^q)]\le C_q$ for any $q\in(0,\tfrac{d}{2d+2})$ and $x\in\Z^d$;
\item\label{item:gkrt-2nd-der}
 $\mb E[\exp(c|\nabla^2\krt(x)|^r)]\le C_r$ for any $r\in(0,\tfrac{d}{2})$ and $x\in\Z^d$;
\item\label{item:gkrt-4}(Stationarity properties)
\begin{itemize}
\item When $d\ge 5$,  the field $\{\krt_\omega(x):x\in\Z^d\}$ is stationary (under $\mb P$);
\item When $d\ge 3$, we have $\nabla\krt_\omega(x)=\nabla\krt_{\theta_x\omega}(0)$ for all $x\in\Z^d$, and the field $\{\nabla\krt_\omega(x):x\in\Z^d\}$  is stationary;
\item When $d=2$, we have $\nabla^2\krt_\omega(x)=\nabla^2\krt_{\theta_x\omega}(0)$ for all $x\in\Z^d$, and the field $\{\nabla^2\krt_\omega(x):x\in\Z^d\}$  is stationary.
\end{itemize}
\end{enumerate}
Moreover,  such a corrector $\krt$ is unique up to an additive constant when $d\ge 3$,  and is unique up to an affine transformation when $d=2$.
\end{thmx}

\begin{thmx}
\label{thm:opt-quant-homo}
Assume {\rm(A1)}, {\rm(A2)}, and {\rm(A3)}. Let $\bar a :=E_{\Q}[\frac{\omega(0)}{\tr\omega(0)}]>0$ and $\bar\psi :=E_{\Q}[\frac{\zeta}{\tr\omega(0)}]$. Suppose $f,g$ are both in $C^4(\R^d)$. Then, for any $\error\in(0,1)$ and $R\ge 2$, there exists a random variable $\ms Y=\ms Y(R,\error, \omega)>1$ with $\mb E[\exp(\ms Y^{{d}/{(2d+2)}-\error})]<C$ such that the solution $u$ of \eqref{eq:elliptic-dirich} satisfies
\[
\max_{x\in B_R}|u(x)-\bar{u}(\tfrac{x}{R})|
\lesssim
\left\{
\begin{array}{rl}
(\log R)^{2+\error}R^{-1}\norm{\bar u}_{C^4(\bar\B_1)}\ms Y &\text{ when }d=2,\\
R^{-1}\norm{\bar u}_{C^4(\bar\B_1)}\ms Y &\text{ when }d\ge 3,
\end{array}
\right.
\]
where $\bar u$ is the solution of the Dirichlet problem \eqref{eq:effective-ellip}.
\end{thmx}

We point out that Theorem \ref{thm:opt-quant-homo} also remains valid for certain finite range dependent environments that are not i.i.d.\  (for example, the ``checkerboard" environment considered in \cite{AL-17}), provided that an Efron--Stein-type concentration inequality holds. Indeed, to obtain the heat-kernel bounds with exponential integrability \cite{GT-22} and the optimal convergence rates \cite{GT-23}, the only requirements are an Efron--Stein-type concentration inequality and a finite range of dependence.
We refer the reader to \cite{CaSu10, AS-14, AL-17, BCDG-18, GPT-19, GT-22, AFL-22, JZ-23,GT-23} for recent developments of quantitative stochastic homogenization of non-divergence form elliptic PDE in both discrete and continuous settings.

\subsection{Main results}

\begin{theorem}
\label{thm:opt-iid-homo}
Assume {\rm(A1), (A2), (A3)}, and $d\ge 3$. Let $\bar a,\bar\psi$ be as in Theorem~\ref{thm:opt-quant-homo}. Suppose that $f\in C^{4,\alpha}(\bar\B_1)$ and $g\in C^{6,\alpha}(\partial\B_1)$ for some $\alpha\in (0,1)$. Then, there exists a $C^6(\bar\B_1)$-extension of $g$ such that the solutions  $u, \bar u$ of \eqref{eq:elliptic-dirich}, \eqref{eq:effective-ellip} satisfy
\[
\max_{x\in B_R}|u(x)-\bar u(\tfrac{x}{R})|
\lesssim
\left\{
  \begin{array}{rl}
  	R^{-3/2}\norm{\bar u}_{C^6(\bar\B_1)}\tx &\text{ when } d=3,\\
  	R^{-2}(\log R)\norm{\bar u}_{C^6(\bar\B_1)}\tx&\text{ when }d\ge 4.
  \end{array}
\right.
\]
An example of such an extension is $g=\bar u$.
\end{theorem}

In the i.i.d. setting with well-prepared boundary conditions, Theorem \ref{thm:opt-iid-homo} yields improved convergence rates compared to the rate $O(R^{-1})$ in Theorem \ref{thm:opt-quant-homo} when $d\geq 3$.
Our approaches and contributions are as follows.
We employ the two-scale asymptotic expansion to obtain the estimate \eqref{eq:two-scale-further} in Proposition \ref{lem:two-scale-exp}.
Due to the i.i.d. structure, the environment enjoys a reflection symmetry; that is, the random fields $\omega(\cdot)$ and $\omega(-\,\cdot)$ are identically distributed. 
(Of course, this symmetry fails for environments with only a finite range of dependence.)
This reflection symmetry implies that the third-order homogenized tensor vanishes, which in turn suggests that a homogenization error bound sharper than
 $O(R^{-1})$ may be attainable.

We remark that such improved rates should not be expected in environments with correlations even slightly weaker than in the point-wise i.i.d.\ case, since the distribution's reflective symmetry is no longer present.

In our i.i.d.\ setting, without the presence of the third-order tensor, the homogenization rate is completely determined by the $C^{0,1}$-bounds of higher-order correctors. 
Roughly speaking, the (first-order) correctors “correct” local coefficients, while the higher-order correctors account for the derivatives of the lower-order correctors. 
Hence, unlike the correctors themselves, the source terms in the equations for higher-order correctors are highly nonlocal, and thus the large-scale regularity estimates in Theorem~\ref{thm:c11}, which apply only to equations with {\it local} source terms, no longer hold.

Our improved convergence rates in Theorem \ref{thm:opt-iid-homo} are inspired by earlier results in the periodic setting \cite{GST-22}, where it was proved that if the environment has ``one degree of freedom" $a(x)$, i.e., the coefficient matrix is of the form $\omega(x)=C+a(x)M$ for some constant symmetric matrices $C$ and $M$, then the third-order homogenized tensor vanishes, leading to an improved rate of $O(R^{-2})$.
The periodic setting, however, is quite rigid.
We provided examples in \cite{GST-22} that show that if the environment has two degrees of freedom, then the third-order homogenized tensor does not vanish in general, and the optimal convergence rate in such cases is only $O(R^{-1})$.

In contrast, in our random setting, $\omega(x)=\mathrm{diag}[\omega_1(x),\ldots, \omega_d(x)]$ has $d$ degrees of freedom, which is fundamentally different from the periodic case considered in \cite{GST-22}.
Moreover, in the periodic setting, the $C^{0,1}$-boundedness of higher-order correctors is automatically guaranteed by periodicity.

One of the main contributions of our work is to establish $C^{0,1}$ bounds for higher-order correctors in the random setting. Our approach is based on comparing the true higher-order correctors with certain {\it localized} higher-order correctors that possess good regularity properties and {\it approximate} higher-order correctors that are stationary.

\begin{remark}
	When $d=2$, the corrector $\upsilon^k$ grows super-linearly (Theorem~\ref{thm:global_krt}). Consequently, via the expansion \eqref{eq:two-scale-further}, it produces an error term of order $R^{-2}|v^k|$ that dominates the desired $O(R^{-1})$ scale. Therefore, in two dimensions, we should not expect an improvement beyond the $O(R^{-1})$ rate in i.i.d.\ environments.
    
    All in all, this demonstrates that, in the absence of additional structural assumptions on the random environment, the boundary conditions, or the dimension, the best convergence rate that one can expect is $O(R^{-1})$, as established in \cite{GT-23}. We refer to Remark~\ref{rmk:optimality} for a discussion of the optimality of the $O(R^{-1})$ rate.
\end{remark}

\section{Preliminaries}
\subsection{The two-scale asymptotic expansion}

In this section, we will investigate the rate of convergence of the solution of the Dirichlet problem \eqref{eq:elliptic-dirich}.  With $a$ as defined in \eqref{eq:def-a},  problem \eqref{eq:elliptic-dirich} is equivalent to
\begin{equation}\label{eq:the-prob}
\left\{
\begin{array}{lr}
L_\omega u(x)=\tfrac 12\tr\big(a(x)\nabla^2u(x)\big)=\frac{1}{R^2}f\big(\tfrac{x}{R}\big)\psi(\theta_x\omega) & x\in \inB R,\\[5 pt]
u(x)=g\big(\tfrac{x}{R}\big) & x\in \partial \inB R,
\end{array}
\right.
\end{equation}
where  $\psi(\omega):=\tfrac{\zeta(\omega)}{\tr\omega(0)}$. With abuse of notation, we also write
\[
\psi(x):=\psi(\theta_x\omega)\quad\text{ for } x\in\Z^d, \qquad\quad \bar\psi:=E_\Q\psi.
\]
Recall $a$ and $\bar a$ from \eqref{eq:def-a} and \eqref{eq:def-abar}. 
For $\omega\in\Omega$ and $k\in\{1,\ldots,d\}$, let $v^k,\xi:\Z^d\to\R$ be solutions of
\begin{align}
L_\omega v^k(x)&=\tfrac12(a_k(x)-\bar a_k)\quad \text{ for } x\in\Z^d,\label{eq:def-vk}\\
L_\omega\xi(x)&=\psi(x)-\bar\psi\qquad\;\;\; \text{for } x\in\Z^d,\label{eq:def-xi}
\end{align}
respectively. That is,  $v^k$ and $\xi$ are global correctors (whose existence is guaranteed by Theorem~\ref{thm:global_krt}) associated to $\tfrac12a_k(x)=\omega(x,x+e_k)$ and $\psi(\omega)$, respectively. We introduce
\begin{align}
&\lambda^k_j(x)=\lambda^k_j(x;\omega) := a_j(x)\left[v^k(x+e_j)-v^k(x-e_j)\right],\;\;
\bar{\lambda}^k_j(x) :=E_{\mb P}\left[\rho(x)\lambda^k_j(x)\right],\label{eq:def-lambda}\\
&\eta_j(x)=\eta_j(x;\omega)\; :=a_j(x)\left[\xi(x+e_j)-\xi(x-e_j)\right],
\quad\;\;\,
\bar{\eta}_j(x) :=E_{\mb P}\left[\rho(x)\eta_j(x)\right],\label{eq:def-eta}
\end{align}
for $x\in\Z^d$. Note that when $d\ge 3$,  by Theorem~\ref{thm:global_krt},  both $\nabla v^k$ and $\nabla\xi$ are stationary fields (under $\mb P$) and so $\bar\lambda^k_j$ and $\bar\eta_j$ are constants that do not depend on $x$.

\begin{proposition}
\label{lem:two-scale-exp}
Suppose that $d\geq 3$, $f\in C^{4,\alpha}(\bar\B_1)$, and $g\in C^{6,\alpha}(\partial\B_1)$. Let $u,\bar u$ be the solutions of \eqref{eq:elliptic-dirich}, \eqref{eq:effective-ellip}, respectively. Recall $a,\bar a$ from \eqref{eq:def-a}, \eqref{eq:def-abar}, and let $v^k,\xi,\lambda^k_j, \bar\lambda^k_j,\eta_j,\bar\eta_j$ be as above.  We extend $g$ so that $g=\bar{u}$ on $\bar\B_1$. Let $\bar z:\bar\B_1\to\R$ be the solution of 
\[
\left\{
\begin{array}{lr}
\frac{1}{2}\tr(\bar{a}D^2 \bar{z}(x))=\tfrac{1}{2}\bar{\lambda}^k_i\partial_{kki}\bar{u}(x)-\tfrac{1}{2}\bar\eta_j\partial_j f(x) & x\in\B_1,\\
\bar z(x)=0 &x\in\partial\B_1.
\end{array}
\right.
\]
For $\omega\in\Omega$, $R>1$, and $j,k\in\{1,\ldots,d\}$, let $p^k_j, s_j:\Z^d \rightarrow \R$ be solutions of 
\begin{align}
L_\omega p^k_j(x)&=\lambda^k_j(x)-\bar\lambda^k_j \quad\text{ for }x\in B_R,\label{eq:def-p}\\
L_\omega s_j(x)&=\eta_j(x)-\bar\eta_j \quad\;\text{ for }x\in B_R\label{eq:def-sj},
\end{align}
respectively. Then, for $\omega\in\Omega$, $R>1$, we have
\begin{align}\label{eq:two-scale-further}
&\max_{x\in B_R}\left|u(x)-\bar{u}(\tfrac xR)-\tfrac 1R\bar z(\tfrac xR)\right|\nn\\
&\qquad\lesssim
\tfrac 1{R^2}\norm{\bar u}_{C^6(\bar{\B}_1)}\max_{i,j,k=1}^d
\bigg[1+\left\| |\nabla p^k_j|+|\nabla s_j|+(|\bar\lambda_j^i|+|\bar\eta_j|)|\nabla v^k|\right\|_{d;B_R}
\nn\\
&\qquad
\qquad\qquad\qquad\qquad\quad\;\;\,+\max_{\bar B_R}\left(|\xi|+|v^k|+\tfrac{1}{R}(|p_i^k|+|s_i|)\right)\bigg].
\end{align}
\end{proposition}

\begin{remark}\label{rmk:optimality}
This proposition identifies $O(R^{-1})$ as the optimal homogenization rate unless $\bar z \equiv 0$, and so no improvement beyond $O(R^{-1})$ should be expected in general. 
         We further expect that Proposition~\ref{lem:two-scale-exp} remains valid for environments with finite range of dependence, for the same reasons discussed in the comments following Theorem~\ref{thm:opt-quant-homo}.
\end{remark}

\begin{proof}
Let $z:\binB R\to\R$ be a solution of the Dirichlet problem
\[
\left\{
\begin{array}{lr}
L_\omega z(x)=\tfrac{1}{R^2}\left[ \tfrac{1}{2}\bar{\lambda}^k_i\partial_{kki}\bar{u}(\tfrac xR)-\tfrac{1}{2}\bar\eta_i\partial_i f(\tfrac xR)\right] & x\in\inB R,\\
z(x)=0 &x\in\partial\inB R.
\end{array}
\right.
\]
Consider the function
\begin{align}\label{eq:def-w}
&w(x)=u(x)-\bar u(\tfrac{x}{R})-\tfrac 1R z(x)\nn\\
&
+\tfrac{1}{R^2}\left[v^k(x)\partial_{kk}\bar u(\tfrac{x}{R})
-f(\tfrac{x}R)\xi(x)\right]- {\tfrac{1}{2}} \tfrac{1}{R^3}\left[\partial_{kki}\bar{u}(\tfrac{x}{R})p_i^k(x)-\partial_i f(\tfrac{x}{R})s_i(x)\right]
\end{align}
for $x\in\binB R\subset\bar\B_R\cap\Z^d$. Then, applying the formula
\begin{align*}
L_\omega(uv)(x)&=u(x)L_\omega v(x)+v(x) L_\omega u(x)+
\sum_{y:y\sim x}\omega(x,y)[u(y)-u(x)][v(y)-v(x)]\\
&=u(x)L_\omega v(x)+v(x) L_\omega u(x) \\ &\qquad+ \tfrac{1}{2}
a_i(x)\left(\nabla_{e_i}u(x)[v(x+e_i)-v(x-e_i)]
+\nabla_i^2u(x)\nabla_{-e_i}v(x)
\right)
\end{align*}
to the products within \eqref{eq:def-w}, we obtain in $\inB R$ that
\begin{align}
&L_\omega[v^k\,\partial_{kk}\bar u(\tfrac{\cdot}{R})]\nn\\
&=
L_\omega v^k \,\partial_{kk}\bar u(\tfrac{\cdot}{R})
+v^k \, L_\omega[\partial_{kk}\bar u(\tfrac{\cdot}{R})]+\tfrac{1}{2}\lambda^k_i\nabla_{e_i}[\partial_{kk}\bar u(\tfrac{\cdot}{R})]+\tfrac{1}{2}a_i \,\nabla_i^2[\partial_{kk}\bar u(\tfrac{\cdot}{R})]\,\nabla_{-e_i}v^k 
\nn\\
&=\tfrac1{2}(a_k -\bar a_k)\partial_{kk}\bar u(\tfrac{\cdot}{R})+\tfrac{1}{2}\tfrac{1}{R}\lambda^k_i\partial_{kki}\bar u(\tfrac{\cdot}{R})+O(\tfrac 1{R^2})\norm{\bar u}_{C^5(\bar\B_1)}\max_{\binB R}\sum_{k = 1}^d|v^k|,\label{eq:term1}
\end{align}
where we used the fact that $|\lambda^k_i|\lesssim|\nabla v^k|\lesssim\max_{\binB R}|v^k|$ in $\inB R$. Similarly,
\begin{align}
L_\omega[f(\tfrac{\cdot}R)\xi ]
&=(\psi-\bar\psi)f(\tfrac{\cdot}R)+{\tfrac{1}{2}}\tfrac{1}{R}\eta_i\partial_i f(\tfrac{\cdot}{R})+O(\tfrac 1{R^2})\norm{\bar u}_{C^5(\bar\B_1)}\max_{\inB R}|\xi|,\label{eq:term2}\\
L_\omega[\partial_{kki}\bar{u}(\tfrac{\cdot}{R})p_i^k ]
&=(\lambda^k_i-\bar\lambda^k_i)\partial_{kki}\bar{u}(\tfrac{\cdot}{R})
+{\tfrac{1}{2}}\tfrac{1}{R}\partial_{kkij}\bar{u}(\tfrac{\cdot}{R})a_j[p^k_i(\cdot+e_j)-p^k_i(\cdot-e_j)]\nn\\
&+O(\tfrac 1{R^2})\norm{\bar u}_{C^6(\bar\B_1)}\max_{\binB R}\sum_{i,k = 1}^d|p^k_i|,\label{eq:term3}\\
L_\omega[\partial_i f(\tfrac{\cdot}{R})s_i ]
&=(\eta_i-\bar\eta_i)\partial_i f(\tfrac{\cdot}{R})
+{\tfrac{1}{2}}\tfrac{1}{R}\partial_{ij}f(\tfrac{\cdot}{R})a_j [s_i(\cdot+e_j)-s_i(\cdot-e_j)]\nn\\
&+O(\tfrac 1{R^2})\norm{\bar u}_{C^6(\bar\B_1)}\max_{\binB R}\sum_{i=1}^d|s_i|.\label{eq:term4}
\end{align}
Note that 
$\abs{L_\omega[\bar u(\tfrac{\cdot}{R})]-\tfrac1{2R^2}\tr[a  D^2\bar{u}(\tfrac{\cdot}{R})]}\lesssim\tfrac1{R^4}\norm{\bar u}_{C^4(\bar\B_1)}$, and 
\begin{align*}
L_\omega u 
=\tfrac1{R^2}f(\tfrac{\cdot}{R})(\psi -\bar\psi)+\tfrac1{R^2}f(\tfrac{\cdot}{R})\bar\psi =\tfrac1{R^2}\big[f(\tfrac{\cdot}{R})L_\omega\xi +
\tfrac1{2}\bar{a}_k\partial_{kk}\bar u(\tfrac{\cdot}{R})\big].
\end{align*}
Then, recalling the definition of $w$ in \eqref{eq:def-w} and writing 
\[
A:=\max_{\binB R}\sum_{i,k=1}^d\left[|\xi|+|v^k|+\tfrac{1}{R}(|p_i^k|+|s_i|)\right],
\]
by \eqref{eq:term1}, \eqref{eq:term2},\eqref{eq:term3}, \eqref{eq:term4},
we have in $\inB R$ that
\begin{align*}
&\abs{L_\omega w } \\
&=
\tfrac1{R^2}\Abs{
f(\tfrac{\cdot}{R})L_\omega\xi +\tfrac1{2}\bar{a}_k\partial_{kk}\bar u(\tfrac{\cdot}{R})-R^2L_\omega[\bar{u}(\tfrac{\cdot}{R})]-{\tfrac{1}{2}}\tfrac{1}{R} \bar{\lambda}^k_i\partial_{kki}\bar{u}(\tfrac{\cdot}{R})+{\tfrac{1}{2}}\tfrac 1R\bar\eta_i\partial_i f(\tfrac{\cdot}{R})
\\
&\qquad+\tfrac1{2}(a_k-\bar a_k)\partial_{kk}\bar u(\tfrac{\cdot}{R})+{\tfrac{1}{2}}\tfrac{1}{R}\lambda^k_i\partial_{kki}\bar u(\tfrac{\cdot}{R})
\nn\\
&\qquad-(\psi-\bar\psi)f(\tfrac{\cdot}{R})-{\tfrac{1}{2}}\tfrac{1}{R}\eta_i\partial_i f(\tfrac{\cdot}{R})+O(\tfrac{1}{R^2})\norm{\bar u}_{C^6(\bar\B_1)}A\\
&\qquad-{\tfrac{1}{2}}\tfrac{1}{R}(\lambda^k_i-\bar\lambda^k_i)\partial_{kki}\bar{u}(\tfrac{\cdot}{R})
-{\tfrac{1}{4}}\tfrac{1}{R^2}\partial_{kkij}\bar{u}(\tfrac{\cdot}{R})a_j[p^k_i(\cdot+e_j)-p^k_i(\cdot-e_j)]\\
&\qquad+{\tfrac{1}{2}}\tfrac{1}{R}(\eta_i-\bar\eta_i)\partial_i f(\tfrac{\cdot}{R})
+{\tfrac{1}{4}}\tfrac{1}{R^2}\partial_{ij}f(\tfrac{\cdot}{R})a_j[s_i(\cdot+e_j)-s_i(\cdot-e_j)]
}
\nn\\
&\lesssim
\tfrac1{R^4}\norm{\bar u}_{C^6(\bar\B_1)}\sum_{i,k=1}^d\left[1+
|\nabla p_i^k|+|\nabla s_i|
+\max_{\binB R}\left(|\xi|+|v^k|+\tfrac{1}{R}(|p_i^k|+ |s_i|)\right)
\right].
\end{align*}	
Thus, by the above inequality, the definition of $w$ from \eqref{eq:def-w}, and the ABP maximum principle, we obtain that
\begin{align}
&\max_{\binB R}|u -\bar u(\tfrac{\cdot}{R})-\tfrac 1R z |\label{eq:ineq1}\\
&\le \max_{\binB R}|w|+
\max_{\binB R}\Abs{\tfrac{1}{R^2}[v^k\partial_{kk}\bar u(\tfrac{\cdot}{R})
-f(\tfrac{\cdot}{R})\xi ]
-{\tfrac{1}{2}}\tfrac{1}{R^3}[\partial_{kki}\bar{u}(\tfrac{\cdot}{R})p_i^k -\partial_i f(\tfrac{\cdot}{R})s_i ]}\nn\\
&\lesssim
R^2\norm{L_\omega w}_{d;\inB R}
+\tfrac1{R^2}\norm{\bar u}_{C^3}\max_{\binB R}\left(|\xi|+|v^k|+\tfrac{1}{R}(|p_i^k|+ |s_i|)\right)\nn
\\&
\lesssim
\tfrac1{R^2}\norm{\bar u}_{C^6}\sum_{i,k=1}^d
\big[1+
\left\||\nabla p_i^k|+|\nabla s_i|\right\|_{d;B_R}+\max_{\binB R}\left(|\xi|+|v^k|+\tfrac{1}{R}(|p_i^k|+|s_i|)\right)
\big].\nn
\end{align}
Furthermore, by \cite[Lemma~30]{GT-23}, 
\[
\max_{\binB R}\abs{z-\bar z(\tfrac{\cdot}{R})}
\lesssim
\tfrac1{R^2}\max_{j,\ell=1}^d(|\bar\lambda_j^\ell|+|\bar\eta_j|)\norm{\bar u}_{C^5(\bar\B_1)}
\sum_{k=1}^d(1+R \norm{\nabla v^k}_{d;\inB R}+\osc_{\binB R}v^k).
\]
This inequality, together with \eqref{eq:ineq1}, yields the claimed bound \eqref{eq:two-scale-further}.
\end{proof}

\subsection{Homogenization in i.i.d.  environments}
In order for the homogenization error $|u(\cdot)-\bar u(\tfrac{\cdot}{R})|$ to be smaller than $O(R^{-1})$,  it is essential (by \eqref{eq:two-scale-further} of Proposition~\ref{lem:two-scale-exp}) that $\bar z\equiv 0$. We recall that in the {\it periodic environment} studied in \cite{GST-22}, the tensor $\bar\lambda_j^k$ (corresponding to $c_j^{kk}$ in \cite[(1.6)]{GST-22}, where $c_j^{kl}$ is the third-order homogenized tensor) vanishes when the coefficient matrix is of the form $\omega(x)=C+a(x)M$ for some $a\in C^{0,\alpha}(\mathbb T^d,\mathbb R)$ and constant symmetric matrices $C, M$. 

In the stochastic setting considered here, we will show that the tensors $\bar \lambda_j^k, \bar\eta_j$ are identically zero when the environment is i.i.d., implying that $\bar z\equiv 0$ and suggesting that a homogenization error bound sharper than 
 $O(R^{-1})$ may be achievable.

The proof relies crucially on the reflection symmetry of the i.i.d. environment, namely that the environments $\omega(\cdot)$ and $\omega(-\,\cdot)$ are identically distributed. Of course, such symmetry does not hold for environments with only a finite range of dependence.
\begin{proposition}
\label{prop:zero}
Recall $\bar \lambda_j^k$ and $\bar\eta_j$ from \eqref{eq:def-lambda} and \eqref{eq:def-eta}, respectively. When $d\ge 3$, we have
\[
\bar\lambda_j^k=\bar\eta_j=0 \quad
\forall j,k \in \{1,\ldots,d\}.
\]
\end{proposition}

\begin{proof}
We will only prove that $\bar\eta_j=0$. The proof of $\bar\lambda_j^k=0$ follows the same lines. 

For $\omega\in\Omega$, let its ``reflection" $\tilde{\omega}=\tilde{\omega}_\omega\in\Omega$ be defined by
\[
\tilde{\omega}(x)=\omega(-x)\quad\text{for } x\in\Z^d.
\]
Then  $\tilde{\omega}$ and $\omega$ have the same distribution under $\mb P$. Note that
\[
\rho_{\tilde{\omega}}(0)=\rho_\omega(0).
\]

Let $\xi=\xi_\omega$ (whose existence is guaranteed by Theorem~\ref{thm:global_krt}) be as defined in \eqref{eq:def-xi}.  Since $\psi$ is a function of $\omega(0)$,  we have $\psi_{\tilde{\omega}}(x)=\psi_\omega(-x)$. Hence,  for $\mb P$-almost all $\omega$, we have
\[
\xi_{\tilde{\omega}}(x)=\xi_\omega(-x)  \quad\forall x\in\Z^d,
\]
up to an additive constant (which may depend on $\omega$).

By the above discussions and the fact that $\nabla\xi$ is a stationary field under $\mb P$, we have for any $j\in \{1,\ldots,d\}$ that
\begin{align*}
\bar\eta_j
&=E_{\mb P}\left[\rho_\omega(0)
a_j^\omega(0)[\xi_\omega(e_j)-\xi_\omega(-e_j)]\right]\\
&=E_{\mb P}\left[\rho_{\tilde\omega}(0)
a_j^{\tilde\omega}(0)[\xi_{\tilde\omega}(-e_j)-\xi_{\tilde\omega}(e_j)]\right]\\
&=-\bar\eta_j,
\end{align*}
and therefore $\bar\eta_j=0$. 
\end{proof}

\section{Bounding $p_j^k$}

In this section, we will compute the bound of a solution $p_j^k$ of \eqref{eq:def-p}. To this end, we construct $p_j^k$ using the ``local Green function" in \cite{GT-23}.

We fix a smooth function $\eta_0\in C^\infty(\R^d)$ that satisfies
\[
\eta_0(\R^d)\subset[0,1], \qquad
\eta_0|_{\R^d\setminus\B_{8/3}}=1,\qquad
\eta_0|_{\B_{7/3}}=0.
\]
For $R>0$, we introduce $\eta=\eta_R :\R^d \rightarrow \R$ given by
\begin{equation}
\label{eq:eta-R}
\eta(x)=\eta_R(x):=\eta_0(\tfrac{x}{R})\quad\text{for } x\in\R^d,
\end{equation}
and we note that 
\[
  \eta|_{\R^d\setminus\B_{3 R}}=1,\qquad \eta|_{\B_{2R}}=0,\qquad 
|D\eta|\lesssim\tfrac1{R},\qquad 
|D^2\eta|\lesssim\tfrac1{R^2}.
\]

 For any fixed $z\in\Z^d$, let $G_R(\cdot,z) = G_R(\cdot,z\,;\,\omega):\Z^d \rightarrow \R$ be the solution to
\[
  L_{\omega}G_R(\cdot,z) = \frac{1}{R^2}\,G_R(\cdot,z)\,\eta - \mathbbm{1}_z\quad\text{in }\Z^d,
\]
where $\mathbbm{1}_z(x) := 1$ if $x = z$, and $\mathbbm{1}_z(x) := 0$ if $x \in \Z^d\setminus\{ z\}$.
Note that $G_R(\cdot, z)$ is $L_\omega$-harmonic on $B_{2R}\setminus\{z\}$, that is, $w(x)=G_R(x,z)$ satisfies

\begin{equation}\label{eq:gr-is-harmonic}
    L_\omega w(x)=0 \quad\text{ for }x\in B_{2R}\setminus\{z\}.
\end{equation}

\begin{lemma}\label{lem:nabla2-gr}
	Let $R>1$ and $d\ge 3$. For $x\in B_R$ and $z\in\Z^d$, we have
	\[
  |\nabla^2G_R(x,z)|\lesssim \mathscr H^{d+1}_{x,z}(|x-z|\wedge R+1)^{-2}\,e^{-c|x-z|/R}\,(|x-z|+1)^{2-d}.
\]
Here $\mathscr H_{x,y}=\mathscr H_x+\mathscr H_y$, and $\mathscr H_x(\omega)=\mathscr H(\theta_x\omega)$ is an exponentially integrable random variable.
\end{lemma}

\begin{proof}
	By \cite[Proposition~36(b)]{GT-23}, we have
	\begin{equation}\label{eq:prop36b}
  |\nabla G_R(x,z)|\lesssim\mathscr H_{x,z}^d(|x-z|\wedge R+1)^{-1}\,e^{-c|x-z|/R}\,(|x-z|+1)^{2-d}.
\end{equation}
Since $|\nabla^2 G_R(x,z)|\lesssim|\nabla G_R(x,z)|$, the statement is true when $|x-z|\wedge R\le \mathscr H_{x,z}$.
It remains to prove the lemma for the case $|x-z|\wedge R>\mathscr H_{x,z}$. 

By \eqref{eq:gr-is-harmonic}, for any $x\in B_R$,  $G_R(\cdot, z)$ is $L_\omega$-harmonic on $B_{|x-z|\wedge R}(x)$.
Hence, when $|x-z|\wedge R>\mathscr H_{x,z}$, we have by the $C^{1,1}$-estimates in \eqref{eq:c2} of Theorem \ref{thm:c11} that 
\begin{align*}
  |\nabla^2 G_R(x,z)|&\lesssim
  \left(\frac{\mathscr H_x}{|x-z|\wedge R}\right)^2\norm{G_R(\cdot, z)}_{1;B_{2R}}\\
&\lesssim \left(\frac{\mathscr H_x}{|x-z|\wedge R}\right)^2\mathscr H_z^{d-1}\,e^{-c|x-z|/R}\,(|x-z|+1)^{2-d},
\end{align*}
where we used \cite[Proposition 36(a)]{GT-23} in the last inequality. The lemma follows.
\end{proof}

Let $\omega'$ denote an independent copy of $\omega$,  and let $\omega_y'\in \Omega$ be the environment defined by 
\begin{equation}\label{eq:def-omegay}
    \omega_y'(x):=\left\{
   \begin{array}{rl}
   	\omega(x) &\text{ if } x\neq y\\
   	\omega'(y)&\text{ if } x=y.
   \end{array}
  \right.
\end{equation}

For any function $\zeta(\omega)$, let 
\[
  \partial_y'\zeta(\omega):=\zeta(\omega_y') - \zeta.
\]

\begin{lemma}\label{lem:u-ver-der}
	If $u(x)=u_\omega(x)$ solves $L_\omega u(x)=\xi(x,\omega)$ for some $\xi:\Z^d \times \Omega \rightarrow \R$, then $\partial'_yu$ solves
	\[
  L_\omega (\partial'_yu)(x)=\partial'_y\xi(x,\omega)-\frac{1}{2}\tr(\partial'_ya(y)\nabla^2u'_y(y))\mathbbm{1}_{y}(x),
\]
where $u_y'(x):=u_{\omega_y'}(x)$, and $\omega_y'\in \Omega$ is as in \eqref{eq:def-omegay}.
\end{lemma}

Let
\begin{equation}\label{eq:p-local}
  p_j^k(x) = -\sum_{z\in \Z^d} G_R(x,z)  \left(\lambda_j^k(z)-\bar{\lambda}_j^k\right).
\end{equation}
We may omit the indices $k,j$ and write $p_j^k, \lambda_j^k, \bar\lambda_j^k$ as $p, \lambda,\bar\lambda$, since our arguments work for all $k,j$. Note that $p$ solves the equation
\[
  L_\omega p=\frac{1}{R^2}p\eta+(\lambda-\bar\lambda) \quad\text{ on }\Z^d.
\]

\begin{proposition}
		\label{prop:size-ep}
		For any $x\in B_R$, we have
	\[
  |p(x)|\lesssim_{\tx_x}\Gamma(R):=\left\{
    \begin{array}{rl}
    	R^{3/2}&\text{ if }d=3,\\
    	R(\log R)^{1/2}&\text{ if }d=4,\\
    	R&\text{ if }d\ge 5.
    \end{array}
  \right.
\]
\end{proposition}

To estimate the size of $p$, we introduce a stationary version of $p$. Let
\begin{equation}\label{eq:def-pap}
  p^\apx(x)=p_\omega^\apx(x)=-\int_0^\infty e^{-s/R^2}E_\omega^x[\lambda(Y_s)-\bar\lambda]\,\dd s
\end{equation}
where $(Y_s)_{s\ge 0}$ is the continuous-time RWRE from Definition~\ref{def:rwre-discrete}. Then $p^\apx$ solves

\begin{equation}\label{eq:lw-p-ap}
    L_\omega p^\apx=\frac{1}{R^2}p^\apx+(\lambda-\bar\lambda) \quad\text{ on }\Z^d.
\end{equation}

Note that, denoting the Green's function $G^\apx=G^\apx_{\omega,R}$ associated to the operator \eqref{eq:lw-p-ap} by 
\begin{equation}\label{eq:def-gapx}
G^\apx(x,y):=\int_0^\infty e^{-s/R^2}p_s^\omega(x,y)\,\dd s, 
\end{equation}
we have a similar formula to \eqref{eq:p-local} for $p^\apx$:
\begin{equation}\label{eq:pap-in-green}
p^\apx(x)=-\sum_{y\in \Z^d}G^\apx(x,y)(\lambda(y)-\bar\lambda).
\end{equation}
Moreover, for $d\ge 3$, we have (for some stretched-exponetially integrable $\tx$)
\begin{equation}\label{eq:green-bound}
    G^\apx(x,y)\lesssim \tx_y\, e^{-c|x-y|/R}(|x-y|+1)^{2-d} \quad\text{ for any }x,y\in\Z^d
\end{equation}
and the same bound holds for $G_R(x,y)$ as well, see \cite[(68) and Proposition 36(a)]{GT-23}.

The function $p^\apx$  is stationary (i.e., $p^\apx(x)$ has the same distribution as $p^\apx(0)$). We claim that there exists a random variable $\tx$ (which may depend on $R$) with $\mb E[\exp(\tx^c)]<C$ such that
\begin{equation}\label{eq:papxsize}
  |p^\apx|\le R^2\tx, \quad |p|\le R^2\tx.
\end{equation}
To this end, observe that by \eqref{eq:pap-in-green}, the definition of $\lambda$ in \eqref{eq:def-lambda}, and Theorem~\ref{thm:global_krt}\eqref{item:gkrt-3},
\begin{equation}\label{eq:pap-bound-by-h}
|p^\apx|\lesssim \sum_{y\in \Z^d}\tx_y\, e^{-c|y|/R}(|y|+1)^{2-d}.
\end{equation}
Recall that the meaning of $\tx$ may differ from line to line.
In fact,  \eqref{eq:pap-bound-by-h} implies that
\begin{equation}
    \label{eq:pap-jensen}
    |p^\apx|\lesssim R^2 \tx.
\end{equation}
Indeed, since $\sum_{y\in \Z^d}e^{-c|y|/R}(|y|+1)^{2-d}\lesssim R^2$, by Jensen's inequality and \eqref{eq:pap-bound-by-h},
\[
\left(\frac C{R^2}|p^\apx|\right)^n\lesssim\frac{1}{R^2}\sum_{y\in \Z^d}e^{-c|y|/R}(|y|+1)^{2-d}\tx_y^n \quad \text{ for any }n\ge 1.
\]
Using the stationarity of the measure $\mb P$ under spatial-shifts, we obtain
\[
\mb E\left[\left(\frac C{R^2}|p^\apx|\right)^n\right]\lesssim\frac{1}{R^2}\sum_{y\in \Z^d}e^{-c|y|/R}(|y|+1)^{2-d}\mb E[\tx_y^n]\asymp\mb E[\tx^n] \quad \text{ for any }n\ge 1.
\]
Therefore $\mb E\bigl[\exp\bigl((\tfrac C{R^2}|p^\apx|)^c\bigr)\bigr]\lesssim\mb E[\exp(\tx^c)]<C'$ and \eqref{eq:pap-jensen} follows. The second bound in \eqref{eq:papxsize} is proved exactly the same way. Claim \eqref{eq:papxsize} is proved.

However, $p^\apx$ does not have the desired regularity. 
We overcome this challenge by using the fact that $p$ has better regularity than $p^\apx$ (although $p$ is only ``almost stationary"), and then comparing $p$ to $p^\apx$.

\begin{remark}\label{rmk:compare-p-pap}
Some comments are in order.
\begin{itemize}
 \item 	Our control of the term $p=p_j^k$ follows a different strategy as the bound of the local corrector $\phi$ in \cite{GT-23}. Recall that, in \cite{GT-23}, we first estimate the approximate corrector $\phi^\apx$ using a sensitivity argument, and then compare $\phi$ to $\phi^\apx$. Such a strategy does not work for reasons explained below.
 \item In \cite[Theorem~17]{GT-23},  $\phi^\apx$ has the desired  $C^{1,1}$-bound because $\psi$ is a {\it local function}. This $C^{1,1}$ bound is crucial in the application of the sensitivity argument (the Efron--Stein inequality). However, for $p^\apx$, the ``source" term $\lambda-\bar\lambda$ is highly non-local, and so the argument above fails. Thus, we are forced to estimate $p$ directly. The control of $p$ consists of two parts:
       \begin{itemize}
  \item the ``random error", i.e., the size of $|p-\mb E[p]|$, which we will achieve using the Efron--Stein inequality;
  \item the ``deterministic error", i.e., the size of $\mb E[p]$, which will be achieved by the triangle inequality $|\mb E[p(x)]|\leq \mb E[|p^\apx|]+\mb E[|p^\apx-p|]$.
\end{itemize} 
For the first estimate, we note the crucial role of the {\it local} Green function $G_R$ in the construction of $p$ as we need its $C^{0,1}$ and $C^{1,1}$ bounds.
For the second estimate, to control $\mb E[|p^\apx|]$, we compare it with $\mb E_\Q[|p^\apx|]$.
And for the term $\mb E[|p^\apx-p|]$, we again use $G_R$ to get the desired bound.
\end{itemize}

\end{remark}

\begin{proof}[Proof of Proposition~\ref{prop:size-ep}]
As explained in Remark~\ref{rmk:compare-p-pap}, we will compare the $C^{1,1}$ regularity of $p^\apx$ to that of $p$. The $C^{1,1}$ bounds will later be used in the sensitivity estimates of $p$ and $p^\apx$. 
The idea is to exploit both the stationarity of $p^\apx$ and the good regularity of $p$.
We divide the proof into several steps.

{\bf Step 1.} By the discrete Krylov--Safonov H\"older estimate (\cite[Corollary
4.6]{KT-96}) and \eqref{eq:papxsize}, we have for $1\le r<R$ that
\begin{align*}
  \osc_{B_r}p^\apx\lesssim \left(\frac{r}{R}\right)^\gamma\left[\norm{p^\apx}_{1,B_R}+R^2\left\|\frac{1}{R^2}p^\apx+(\lambda-\bar\lambda)\right\|_{d,B_R}\right]\stackrel{\eqref{eq:papxsize}}\lesssim \left(\frac{r}{R}\right)^\gamma R^2\tx,
\end{align*}
where $\gamma>0$ depends only on $d$ and the ellipticity constant. Hence, 
\[
  [p^\apx]_{\gamma;B_{R/2}}\lesssim R^{2-\gamma}\tx.
\]
Similarly, we find that $[p]_{\gamma;B_{R/2}}\lesssim R^{2-\gamma}\tx$.

{\bf Step 2.} Note that $p-p^\apx$ solves
\[
  L_\omega (p-p^\apx)=\frac{1}{R^2}(p\eta-p^\apx).
\]
Here, we consider the difference $p-p^\apx$ to remove the nonlocal source term $\lambda-\bar\lambda$.
 By the $C^{1,1}$ estimate \eqref{eq:c2} with $u=p-p^\apx$, $\psi\equiv 0$, $\sigma=\gamma$, and $f=\frac{1}{R^2}(p\eta-p^\apx)$, we obtain
\begin{equation}\label{eq:nabla2p-1}
  |\nabla^2 (p-p^\apx)|
  \lesssim \left(\tfrac{\tx}{R}\right)^2 \left(R^2+R^{2+\gamma}\left[\frac{1}{R^2}(p\eta-p^\apx)\right]_{\gamma;B_R}\right)\lesssim\tx^2.
\end{equation}
Furthermore, in view of \eqref{eq:p-local}, we have for $x\in B_R$ that
\[
        \nabla^2 p(x) = -\sum_{z\in \Z^d} \nabla^2G_R(x,z) \left(\lambda(z)-\bar{\lambda}\right),
\]
and hence, by Lemma~\ref{lem:nabla2-gr} {(and the argument that leads \eqref{eq:pap-bound-by-h} to \eqref{eq:pap-jensen}),}
 \begin{equation}\label{eq:nabla2p-2}
   |\nabla^2 p(x)|\lesssim \tx_x\sum_{z\in\Z^d}(|z|\wedge R+1)^{-2}e^{-c|z|/R}(|z|+1)^{2-d}\lesssim (\log R) \tx_x.
\end{equation}

We note that this argument for $\nabla^2p$ does not work for $\nabla^2 p^\apx$ because the Green function $G^\apx$ does not have as good regularity as $G_R$.

Combining \eqref{eq:nabla2p-1} and \eqref{eq:nabla2p-2}, we obtain $|\nabla^2 p^\apx(x)|\lesssim (\log R)\tx_x$ for $x\in B_R$. Since $(\nabla^2p^\apx(x))_{x\in\Z^d}$ is stationary, we conclude that, for all $x\in\Z^d$,
\begin{equation}\label{eq:c2bound-p-ap}
    |\nabla^2 p^\apx(x)|\lesssim (\log R)\tx_x.
\end{equation}
This yields that $p^\apx_\omega$ has a close to optimal $C^{1,1}$ regularity.
However, its $C^{0,1}$ bound is not as good as we want.
 
{\bf Step 3.} Next, we deduce a formula for $\partial'_y p^\apx$. By \eqref{eq:lw-p-ap} and Lemma~\ref{lem:u-ver-der},
\begin{equation}\label{eq:250707-1}
    L_\omega(\partial'_yp^\apx)=\frac{1}{R^2}\partial'_yp^\apx +\partial'_y\lambda -\frac{1}{2}\tr(\partial'_ya(y)\nabla^2 p_y^{\apx'} (y))\mathbbm{1}_{y},
\end{equation}
where $p_y^{\apx'}:=p^\apx_{\omega'_y}$ is as in Lemma~\ref{lem:u-ver-der}.
Let us estimate $\partial_y'\lambda$. Using the product rule for $\partial_y'$, and writing $\tilde{v}^k := v^k + v^k(\cdot - e_j)$, we have 
\begin{align}\label{eq:250707-2}
\partial_y' \lambda = \partial_y' ( a_j\nabla_{e_j}\tilde v^k) =a_j^{\omega_y'}\, \nabla_{e_j}\partial_y' \tilde v^k + \mathbbm{1}_y\, \partial_y' a_j \,  \nabla_{e_j} \tilde v^k.
\end{align}
Hence, by \eqref{eq:250707-1} and \eqref{eq:250707-2},
\begin{align}\label{eq: Lom}
   L_\omega(\partial'_yp^\apx)=\frac{1}{R^2}\partial'_yp^\apx +a_j^{\omega_y'}\, \nabla_{e_j}\partial_y' \tilde v^k- \left(\frac{1}{2}\tr(\partial'_ya\,\nabla^2 p_y^{\apx'})-\partial_y' a_j \nabla_{e_j} \tilde v^k\right)\mathbbm{1}_y. 
\end{align}
Notice that the function $G^\apx$ defined in \eqref{eq:def-gapx} satisfies
\[
  L_\omega G^\apx(\cdot,z)=\frac{1}{R^2}G^\apx(\cdot,z)-\mathbbm{1}_{z}\quad\text{in }\Z^d
\]
 for any fixed $z\in \Z^d$. Then, we see from \eqref{eq: Lom} that
\begin{align}\label{eq:formula-p-ap}
  \partial'_yp^\apx(x)&=
  G^\apx(x,y)\left(\frac{1}{2}\tr(\partial'_ya(y)\,\nabla^2 p_y^{\apx'} (y))-\partial_y' a_j(y)\,  \nabla_{e_j} \tilde v^k(y)\right)\nn\\
  &\qquad-\sum_{z\in\Z^d}G^\apx(x,z)\,a_j^{\omega_y'}(z)\, \nabla_{e_j}\partial_y' \tilde v^k(z).
\end{align}
Note that, by the inequality below \cite[(94)]{GT-23},
\begin{align*}
\lvert \nabla_{e_j} \partial'_y \tilde v^k(z) \rvert \lesssim \tx_y'^2 \tx_{y,z}^d\, e^{-c|z-y|/R}\, (\lvert z-y\rvert+1)^{2-d}\, (|z-y|\wedge R+1)^{-1} .
\end{align*}
This inequality, together with \eqref{eq:formula-p-ap}, \eqref{eq:c2bound-p-ap}, and \cite[(68)]{GT-23}, yields
\begin{align}\label{eq:partialpap}
  &|\partial'_yp^\apx(x)|\lesssim_{\tx_{x,y}} 
  e^{-c|x-y|/R}(|x-y|+1)^{2-d}(\log R)\\
  &\;+\sum_{z\in \Z^d} e^{-c(|x-z|+|z-y|)/R}(|x-z|+1)^{2-d}(|z-y|+1)^{2-d}(|z-y|\wedge R+1)^{-1}.\nn
\end{align}
In particular,
\begin{align}\label{T(R) def}
    \sum_{y\in \Z^d}|\partial_y' p^\apx(0)|^2
  &\lesssim_\tx (\log R)^2\sum_{y\in \Z^d} e^{-2c|y|/R}(|y|+1)^{4-2d} \nn \\ &\quad+  \sum_{y\in \Z^d}\left\lvert\sum_{z\in \Z^d}  e^{-c(|z|+|z-y|)/R}    (|z|+1)^{2-d} \frac{(|z-y|+1)^{2-d}}{|z-y|\wedge R+1}\right\rvert^2 \nn \\ &\lesssim_\tx \Gamma(R)^2 +  T(R)^2,
\end{align}
where $\Gamma(R)$ is defined as in the statement of the theorem, and
\begin{align*}
    T(R) := \left(\int_{\R^d}\left\lvert\int_{\R^d}  g_R(x,y)\, \dd x\right\rvert^2 \dd y\right)^{1/2}
\end{align*}
with
\begin{align*}
g_R(x,y) :=e^{-c(|x+y|+|x|)/R}    (|x+y|+1)^{2-d} (|x|+1)^{2-d} (|x|\wedge R+1)^{-1}.
\end{align*}

{\bf Step 4.} 
We claim that
\begin{align}\label{T(R) bd}
  T(R)\lesssim \Gamma(R).
\end{align}
First, noting that $\int_{\B_R} (|x|+1)^{1-d} \dd x \lesssim R$ and using H\"older's inequality, we see that
\begin{align*}
    \int_{\B_{2R}}\left\lvert\int_{\B_{R}}  g_R(x,y)\, \dd x\right\rvert^2 \dd y &\leq 
    \int_{\B_{2R}}\left\lvert\int_{\B_{R}}  (|x+y|+1)^{2-d}(|x|+1)^{1-d}\, \dd x\right\rvert^2 \dd y \\
    &\lesssim R \int_{\B_{2R}} \int_{\B_{R}}  (|x+y|+1)^{4-2d}(|x|+1)^{1-d}\, \dd x\, \dd y \\
    &\lesssim R \int_{\B_{R}} (|x|+1)^{1-d} \int_{\B_{3R}}(|z|+1)^{4-2d} \, \dd z\, \dd x \\
    &\lesssim R^2 \int_0^{3R} (r+1)^{4-2d} r^{d-1}\, \dd r \lesssim \Gamma(R)^2.
\end{align*}
Second, using that $|x+y|\geq |y|-|x|\geq R$ for $(x,y)\in \B_R\times (\R^d \setminus \B_{2R})$, we obtain
\begin{align*}
    &\int_{\R^d\setminus \B_{2R}}\left\lvert\int_{\B_{R}}  g_R(x,y)\, \dd x\right\rvert^2 \dd y \leq R^{4-2d} 
    \int_{\R^d\setminus \B_{2R}}\left\lvert\int_{\B_{R}}  e^{-c|x+y|/R}(|x|+1)^{1-d}\, \dd x\right\rvert^2 \dd y
    \\ 
    &\leq R^{4-2d} \int_{\R^d\setminus \B_{2R}} e^{-2c |y| /R} \left\lvert\int_{\B_{R}}  e^{c|x|/R}(|x|+1)^{1-d}\, \dd x\right\rvert^2 \dd y \\
    &\lesssim R^{4-2d} \left(\int_{2R}^{\infty} e^{-2c r /R} r^{d-1} \,\dd r\right) \left\lvert\int_{0}^R  e^{cr/R}(r+1)^{1-d}r^{d-1}\, \dd r\right\rvert^2  \lesssim R^{6-d} \lesssim \Gamma(R)^2. 
\end{align*}
Third, using that $|x+y|\geq |x|-|y|\geq R/2$ for $(x,y)\in (\R^d \setminus \B_{R}) \times \B_{R/2}$, we obtain
\begin{align*}
    &\int_{\B_{R/2}}\left\lvert\int_{\R^d \setminus \B_{R}}  g_R(x,y)\, \dd x\right\rvert^2 \dd y \\ &\lesssim  R^{4-2d}R^{-2} \int_{\B_{R/2}}\left\lvert\int_{\R^d \setminus \B_{R}}  e^{-c(|x+y|+|x|)/R}(|x|+1)^{2-d}\, \dd x\right\rvert^2 \dd y \\
    &\lesssim R^{2-2d}\int_{\B_{R/2}} e^{ 2c |y|/R}\left\lvert\int_{\R^d \setminus \B_{R}}  e^{-2c|x|/R}(|x|+1)^{2-d}\, \dd x\right\rvert^2 \dd y \\
    &\lesssim R^{2-2d} \left(\int_0^{R/2} e^{ 2c r/R}r^{d-1}\,\dd r\right) \left\lvert \int_R^{\infty} e^{-2cr/R}(r+1)^{2-d} r^{d-1}\,\dd r\right\rvert^2 \lesssim R^{6-d} \lesssim \Gamma(R)^2.
\end{align*}
Finally, noting that $\int_{\R^d\setminus \B_R} e^{-c|x|/R}\dd x \lesssim R^d$ and using H\"older's inequality, we find 
\begin{align*}
    &\int_{\R^d\setminus \B_{R/2}}\left\lvert\int_{\R^d \setminus \B_{R}}  g_R(x,y)\, \dd x\right\rvert^2 \dd y \\
    &\lesssim R^{4-2d} R^{-2} \int_{\R^d\setminus \B_{R/2}}\left\lvert\int_{\R^d \setminus \B_{R}}  e^{-c(|x+y|+|x|)/R}    (|x+y|+1)^{2-d}\, \dd x\right\rvert^2 \dd y \\
    &\lesssim R^{2-2d} R^d \int_{\R^d\setminus \B_{R/2}}\int_{\R^d \setminus \B_{R}}  e^{-c(2|x+y|+|x|)/R}    (|x+y|+1)^{4-2d}\, \dd x\, \dd y \\
    &\lesssim R^{2-d} \int_{\R^d \setminus \B_{R}} e^{-c|x|/R} \int_{\R^d} e^{-2c|z|/R}    (|z|+1)^{4-2d} \,\dd z \,\dd x \\
    &\lesssim R^2 \int_0^{\infty}  e^{-2cr/R}    (r+1)^{4-2d} r^{d-1} \,\dd r \lesssim \Gamma(R)^2.
\end{align*}
Altogether, we obtain the claimed bound \eqref{T(R) bd}, and hence,
\begin{align*}
    \sum_{y\in \Z^d}|\partial_y' p^\apx(0)|^2
  \lesssim_\tx \Gamma(R)^2 + T(R)^2 \lesssim_\tx \Gamma(R)^2.
\end{align*}

{\bf Step 5.} 
 Using an argument similar to \eqref{eq:formula-p-ap}, we find
\begin{align}\label{eq:dyp-p}
  \partial'_yp(x)&=G_R(x,y)\left(\frac{1}{2}\tr(\partial'_ya(y)\nabla^2 p_y' (y))-\partial_y' a_j (y)\,  \nabla_{e_j} \tilde v^k (y)\right)\nn\\
  &\qquad-\sum_{z\in\Z^d}G_R(x,z)\,a_j^{\omega_y'}(z)\, \nabla_{e_j}\partial_y' \tilde v^k(z)
\end{align}
for $x\in B_R$. Then, using \eqref{eq:nabla2p-2}, arguments similar to steps 3 and 4, and \cite[Proposition 36]{GT-23}, we obtain for any $x\in B_R$ that
\begin{equation}
  \label{eq:ver-der-p}
  \sum_{y\in \Z^d}|\partial_y' p(x)|^2
  \lesssim_\tx\Gamma(R)^2.
\end{equation}
By the $L_p$ version of the Efron--Stein inequality \cite[(38)]{GT-23},  for any $q\ge 2$, 
\begin{align}
  \mathbb{E}\left[ \lvert p(x) - \mathbb{E} p(x)\rvert^q \right] &\lesssim q^{q/2} \mathbb{E}\left[V^{q/2}(p(x))\right]\lesssim 
  q^{q/2}\Gamma(R)^q \quad\forall x\in B_R,\label{eq:var-p}
\\
    \mb E[(p^\apx-\mb Ep^\apx)^q]&\lesssim q^{q/2}\Gamma(R)^q,\label{eq:var-pap}
\end{align}
where $V(p(x)) := \sum_{y\in \Z^d} \lvert \partial_y' p(x)\rvert^2$.

{\bf Step 6.} Observe that $E_\Q p^\apx=0$. 
Hence, we have for the $\mb Q$-variance that
\begin{align*}\label{eq:liaoning-0712}
 E_\Q[(p^\apx)^2]&\le E_\Q[(p^\apx)^2]+\mb E[p^\apx]^2=E_\Q[(p^\apx-\mb Ep^\apx)^2]\\
 &\lesssim \mb E[(p^\apx-\mb Ep^\apx)^{3}]^{2/3}\nn\\
 &\stackrel{\eqref{eq:var-pap}}\lesssim\Gamma(R)^2,
\end{align*}
where in the second inequality, we used H\"older's inequality and the exponential integrability of $\rho=\dd\mb Q/\dd\mb P$ in \eqref{eq:rho-bounds}. 
By H\"older's inequality and a moment bound of $\rho^{-1}=\dd\mb P/\dd\mb Q$, we obtain
\begin{equation}\label{eq:ebound-apap}
  \mb E[|p^\apx|]\lesssim (E_\Q[(p^\apx)^{2}])^{1/2}\lesssim\Gamma(R).
  \end{equation}

{\bf Step 7.} 
Let $u=p^\apx-p$. Then $u$ solves the equation
\[
  L_\omega u=\frac{1}{R^2}u\eta+\frac{1}{R^2}p^\apx(1-\eta).
\]
 Hence, we have for $x\in\Z^d$ that
\begin{align*}
|u(x)|&=\frac{1}{R^2}\Abs{\sum_{y\in B_{3R}}G_R(x,y)p^\apx(y)(1-\eta(y))}\\
&\lesssim_\tx\frac{1}{R^2}\sum_{y\in B_{3R}}e^{-c|x-y|/R}(|x-y|+1)^{2-d}|p^\apx(y)|.
\end{align*}
Using this inequality and \eqref{eq:ebound-apap}, we 
obtain $\mb E[|u(x)|]\lesssim\Gamma(R)$.
Therefore, using \eqref{eq:ebound-apap} again, we obtain for any $x\in\Z^d$ that
\begin{equation}\label{eq:bound-expec-p}
  |\mb E[p(x)]|\lesssim \mb E[|p^\apx|]+\mb E[|u(x)|]\lesssim\Gamma(R).
\end{equation}

{\bf Step 8.}
 The bound \eqref{eq:var-p}, together with \eqref{eq:bound-expec-p}, implies that, for any $x\in B_R$, 
\[
  |p(x)-\mb E[p(x)]|\le \Gamma(R) \tx_x
\]
for a stretched exponentially integrable variable $\tx$.
This inequality, together with \eqref{eq:bound-expec-p}, yields the desired result.
\end{proof}

\section{Bounding $\nabla p_j^k$}
In this section, we will compute the bound of $\nabla p_j^k$ where $p_j^k$ is defined in \eqref{eq:p-local}.
As in the previous section, we omit the indices $k,j$ and write $p_j^k, \lambda_j^k, \bar\lambda_j^k$ as $p, \lambda,\bar\lambda$, respectively, since our arguments work for all $k,j$.

\begin{proposition} For any $x\in B_R$, we have
	\label{prop:grad-p}
\[
  |\nabla p(x)|\lesssim_{\tx_x}\widetilde\Gamma(R):=\left\{
    \begin{array}{rl}
    	R^{1/2}&\text{ if }d=3,\\
    	\log R&\text{ if }d\geq 4.
    \end{array}
  \right.
\]
\end{proposition}
\begin{proof}
Again, we divide the proof into several steps.

	{\bf Step 1.} In view of \eqref{eq:dyp-p}, we have that
\begin{align}\label{eq:250708-1}
  \partial'_y\nabla p(x)&=\nabla_xG_R(x,y)\left(\frac{1}{2}\tr(\partial'_ya(y)\nabla^2 p_y' (y))-\partial_y' a_j (y)\,  \nabla_{e_j} \tilde v^k (y)\right)\nn\\
  &\qquad-\sum_{z\in\Z^d}\nabla_x G_R(x,z)\,a_j^{\omega_y'}(z)\, \nabla_{e_j}\partial_y' \tilde v^k(z),
\end{align}
where the subscript of $\nabla_x$ indicates that $\nabla$ is applied only to the variable $x$. For $x\in B_R$, by \eqref{eq:prop36b}, \eqref{eq:250708-1}, and the same arguments as in \eqref{eq:partialpap}, we obtain
\begin{align*}
|\partial'_y\nabla p(x)|&\lesssim_{\tx_{x,y}}
e^{-c|x-y|/R} \frac{(|x-y|+1)^{2-d}}{|x-y|\wedge R+1} (\log R)\\
\MoveEqLeft+\sum_{z\in \Z^d} e^{-c(|x-z|+|z-y|)/R}\frac{(|x-z|+1)^{2-d}}{|x-z|\wedge R+1}\frac{(|z-y|+1)^{2-d}}{|z-y|\wedge R+1}.\nn
\end{align*}

{\bf Step 2.} In particular, in view of \eqref{T(R) def} and \eqref{T(R) bd}, we have the bound
\begin{align*}
    &\sum_{y\in \Z^d}|\partial_y' \nabla p(0)|^2
  \lesssim_\tx (\log R)^2\sum_{y\in \Z^d} e^{-2c|y|/R}\frac{(|y|+1)^{4-2d}}{(|y|\wedge R + 1)^2} \\ &\qquad\qquad\qquad\qquad+ \sum_{y\in \Z^d}\left\lvert\sum_{z\in \Z^d}  e^{-c(|z|+|z-y|)/R}\frac{(|z|+1)^{2-d}}{|z|\wedge R+1}\frac{(|z-y|+1)^{2-d}}{|z-y|\wedge R+1}\right\rvert^2 \\ 
  &\lesssim_\tx (\log R)^2 + \sum_{y\in \Z^d}\left\lvert\sum_{z\in B_R}  e^{-c(|z|+|z-y|)/R}(|z|+1)^{1-d}\frac{(|z-y|+1)^{2-d}}{|z-y|\wedge R+1}\right\rvert^2 + \frac{1}{R^2} \Gamma(R)^2
  \\&\lesssim_\tx  (\log R)^2 + \int_{\R^d}\left\lvert\int_{\B_R}  \tilde g_R(y,z)\, \dd z\right\rvert^2 \dd y + \frac{1}{R^2} \Gamma(R)^2,
\end{align*}
where
\begin{align*}
\tilde g_R(y,z) :=e^{-c(|z|+|z-y|)/R}(|z|+1)^{1-d}(|z-y|+1)^{2-d}(|z-y|\wedge R+1)^{-1}.
\end{align*}
First, noting that $\int_{\B_{R}} (|x|+1)^{1-d} \dd x \lesssim R$ and $\int_{\B_{3R}\setminus \B_R} (|x|+1)^{2-d} \dd x \lesssim R^2$, we obtain  
\begin{align*}
    &\int_{\B_{2R}}\left\lvert\int_{\B_{R}}  \tilde g_R(y,z)\, \dd z\right\rvert^2 \dd y \leq \int_{\B_{2R}}\left\lvert\int_{\B_{3R}}  \frac{(|x|+1)^{2-d}}{|x|\wedge R+1}\,(|x+y|+1)^{1-d}\, \dd x\right\rvert^2 \dd y \\
    &\lesssim \int_{\B_{2R}}\left\lvert\int_{\B_{R}} (|x|+1)^{1-d}\, (|x+y|+1)^{1-d}\,  \dd x\right\rvert^2 \dd y \\ &\qquad+ \frac{1}{R^2} \int_{\B_{2R}}\left\lvert\int_{\B_{3R}\setminus \B_R} (|x|+1)^{2-d}\, (|x+y|+1)^{1-d}\,  \dd x\right\rvert^2 \dd y \\
    &\lesssim R\int_{\B_{2R}} \int_{\B_{R}}  (|x|+1)^{1-d}\,(|x+y|+1)^{2-2d}\,  \dd x \, \dd y \\ &\qquad+ \int_{\B_{2R}} \int_{\B_{3R}\setminus \B_R}  (|x|+1)^{2-d}\,(|x+y|+1)^{2-2d}\,  \dd x\,  \dd y \\
    &\lesssim R \int_{\B_{R}} (|x|+1)^{1-d} \int_{\B_{3R}}   (|z|+1)^{2-2d}\, \dd z \, \dd x \\ &\qquad+ \int_{\B_{3R}\setminus \B_R} (|x|+1)^{2-d} \int_{\B_{5R}}   (|z|+1)^{2-2d}\, \dd z\,  \dd x \\
    &\lesssim R^{4-d} \lesssim \frac{1}{R^2}\Gamma(R)^2.
\end{align*}
Second, using that $|z-y|\geq |y|-|z|\geq R$ for $(y,z)\in (\R^d \setminus \B_{2R})\times \B_R$, we obtain
\begin{align*}
    &\int_{\R^d\setminus \B_{2R}}\left\lvert\int_{\B_{R}}  \tilde g_R(y,z)\, \dd z\right\rvert^2 \dd y \leq R^{2-2d} 
    \int_{\R^d\setminus \B_{2R}}\left\lvert\int_{\B_{R}}  e^{-c|z-y|/R}(|z|+1)^{1-d}\, \dd z\right\rvert^2 \dd y
    \\ 
    &\leq R^{2-2d} \int_{\R^d\setminus \B_{2R}} e^{-2c |y| /R} \left\lvert\int_{\B_{R}}  e^{c|z|/R}(|z|+1)^{1-d}\, \dd z\right\rvert^2 \dd y \lesssim R^{4-d} \lesssim \frac{1}{R^2}\Gamma(R)^2. 
\end{align*}
Altogether, we find that
\[
  \sum_{y\in \Z^d}|\partial'_y\nabla p(0)|^2\lesssim_{\tx} (\log R)^2 + \frac{1}{R^2} \Gamma(R)^2 \lesssim_{\tx}   \widetilde\Gamma(R)^2.
\]
{\bf Step 3.} The $L_p$ version of the Efron--Stein inequality then yields
\[
  |\nabla p(x)-\mb E[\nabla p(x)]|\lesssim_{\tx_x}\widetilde\Gamma(R) \quad\forall x\in B_R.
\]
Further, using \cite[Lemma~40]{GT-23} (with $\mu(R)$ replaced by the size $\Gamma(R)$ of $p$), we obtain
\[
  |\mb E[\nabla p(x)]|\lesssim \widetilde\Gamma(R).
\]
Therefore, we conclude that
\[
|\nabla p(x)|\lesssim_{\tx_x} \widetilde\Gamma(R) \quad\forall x\in B_R.
\]
\end{proof}

\begin{proof}[Proof of Theorem \ref{thm:opt-iid-homo}]
    Propositions \ref{prop:size-ep} and \ref{prop:grad-p} give the bounds for $\lvert p_j^k\rvert$ and $\lvert \nabla p_j^k\rvert$.
    By repeating the same analysis, we get the corresponding bounds for $\lvert s_j\rvert$ and $\lvert \nabla s_j\rvert$. By combining Proposition \ref{lem:two-scale-exp} with Propositions \ref{prop:zero}, \ref{prop:size-ep}, and \ref{prop:grad-p}, we obtain the desired conclusion.
\end{proof}

\appendix

\end{document}